\documentclass[a4paper, reqno]{amsart}

\usepackage{graphicx}
\usepackage{tikz}
\usetikzlibrary{calc}

\newtheorem{theorem}{Theorem}
\newtheorem{lemma}[theorem]{Lemma}
\newtheorem{proposition}[theorem]{Proposition}
\newtheorem{corollary}[theorem]{Corollary}
\newtheorem{conjecture}[theorem]{Conjecture}

\theoremstyle{definition}
\newtheorem{definition}{Definition}

\newcommand{\btree}{\ensuremath{\beta(0,1)}-tree}
\newcommand{\btrees}{\btree s}
\newcommand{\F}{\mathcal{F}}
\newcommand{\floor}[1]{\lfloor #1\rfloor}

\newcommand{\nodestyle}{
  \tikzstyle{every node} = [font=\footnotesize];
}
\newcommand{\discstyle}{
  \tikzstyle{blk} =
    [ circle,fill=black,draw=black, minimum size=3.7pt, inner sep=0pt ];
  \tikzstyle{wht} =
    [ circle,fill=white,draw=black, minimum size=3.7pt, inner sep=0pt ];
  \tikzstyle{sml} =
    [ circle,fill=black,draw=black, minimum size=3pt, inner sep=0pt ];
}
\newcommand{\ostyle}{
  \tikzstyle{o} =
  [ circle,fill=gray,draw=black, minimum size=3pt, inner sep=0pt ];
}
\newcommand{\style}{
  \nodestyle
  \discstyle
}
\newcommand{\scl}{0.5}
\newcommand{\leaf}{\hspace{0.8pt}
  \begin{tikzpicture}[scale=\scl, semithick, baseline=-2.5pt]
    \style;
    \node[sml] (a) at (0,0) {};
  \end{tikzpicture}\hspace{0.8pt}
}
\newcommand{\trea}[4]{
  \begin{tikzpicture}[scale=\scl, semithick, baseline=(d)]
    \style;
    \node [sml] (a) at (0,0) {};
    \node [sml] (b) at (0,1) {};
    \node [sml] (c) at (0,2) {};
    \node [sml] (d) at (0,3) {};
    \draw 
      (a) node[right=2pt] {#1} --
      (b) node[right=2pt] {#2} --
      (c) node[right=2pt] {#3} --
      (d) node[right=2pt] {#4};
  \end{tikzpicture}
}
\newcommand{\treb}[4]{
  \begin{tikzpicture}[scale=\scl, semithick, baseline=(d)]
    \style;
    \node [sml] (a) at (-0.65,0) {};
    \node [sml] (b) at ( 0.65,0) {};
    \node [sml] (c) at (0,1) {};
    \node [sml] (d) at (0,2) {};
    \draw (a) node[below=2pt] {#1} -- (c);
    \draw (b) node[below=2pt] {#2} -- (c);
    \draw (c) node[right=2pt] {#3} -- (d) node[right=2pt] {#4};
  \end{tikzpicture}
}
\newcommand{\trec}[4]{
  \begin{tikzpicture}[scale=\scl, semithick, baseline=(d)]
    \style;
    \node [sml] (a) at (-0.65,0) {};
    \node [sml] (b) at (-0.65,1) {};
    \node [sml] (c) at ( 0.65,1) {};
    \node [sml] (d) at (   0,2) {};
    \draw (a) node[left=2pt] {#1} -- (b);
    \draw (b) node[left=2pt] {#2} -- (d);
    \draw (c) node[below=1pt] {#3} -- (d) node[right=2pt] {#4};
  \end{tikzpicture}
}
\newcommand{\tred}[4]{
  \begin{tikzpicture}[scale=\scl, semithick, baseline=(d)]
    \style;
    \node [sml] (a) at ( 0.65,0) {};
    \node [sml] (b) at ( 0.65,1) {};
    \node [sml] (c) at (-0.65,1) {};
    \node [sml] (d) at (   0,2) {};
    \draw (a) node[right=2pt] {#1} -- (b);
    \draw (b) node[right=2pt] {#2} -- (d);
    \draw (c) node[below=1pt] {#3} -- (d) node[right=2pt] {#4};
  \end{tikzpicture}
}
\newcommand{\tree}[4]{
  \begin{tikzpicture}[scale=\scl, semithick, baseline=(d)]
    \style;
    \node [sml] (a) at (-1,0) {};
    \node [sml] (b) at ( 0,0) {};
    \node [sml] (c) at ( 1,0) {};
    \node [sml] (d) at ( 0,1) {};
    \draw (a) node[below=1pt] {#1} -- (d);
    \draw (b) node[below=1pt] {#2} -- (d);
    \draw (c) node[below=1pt] {#3} -- (d) node[right=2pt, yshift=2pt] {#4};
  \end{tikzpicture}
}
\newcommand{\figtree}{
  \begin{tikzpicture}[xscale=0.37, yscale=0.5, semithick, baseline=(r)]
    \style;
    \node [sml] (r)       at ( 0,6) {};
    \node [sml] (r1)      at (-3,5) {};
    \node [sml] (r2)      at (-1,5) {};
    \node [sml] (r3)      at ( 1,5) {};
    \node [sml] (r31)     at ( 1,4) {};
    \node [sml] (r4)      at ( 3,5) {};
    \node [sml] (r41)     at ( 3,4) {};
    \node [sml] (r411)    at ( 3,3) {};
    \node [sml] (r4111)   at ( 2,2) {};
    \node [sml] (r41111)  at ( 2,1) {};
    \node [sml] (r411111) at ( 2,0) {};
    \node [sml] (r4112)   at ( 4,2) {};
    \draw (r) node[above=2pt] {4} -- (r1)  node[below left=-1pt] {0}
          (r)                     -- (r2)  node[below left=-1pt] {0}
          (r)                     -- (r3)  node[left=2pt, yshift=-2pt] {1}
          (r)                     -- (r4)  node[right=2pt] {2}
          (r3)                    -- (r31) node[left=2pt] {0}
          (r4)                    -- (r41) node[right=2pt] {1}
          (r41)                   -- (r411) node[right=2pt] {3}
          (r411)                  -- (r4111) node[left=2pt] {2}
          (r4111)                 -- (r41111) node[left=2pt] {1}
          (r41111)                -- (r411111) node[left=2pt] {0}
          (r411)                  -- (r4112) node[right=2pt] {0};
  \end{tikzpicture}
}

\newcommand{\treeS}{
\begin{tikzpicture}[yscale=0.5, xscale=0.35, semithick, baseline=(r1)]
  \style;
  \node [wht] (r)   at ( 0,3) {};
  \node [wht] (r1)  at (-1,2) {};
  \node [wht] (r11) at (-1,1) {};
  \node [wht] (r2)  at ( 1,2) {};
  \draw   (r)   node[above=1pt] {1}
       -- (r1)  node[left=1pt]  {0}
       -- (r11) node[left=1pt]  {0}
   (r) -- (r2)  node[below=1pt] {0};
\end{tikzpicture}
}
\newcommand{\treeT}{
\begin{tikzpicture}[yscale=0.5, xscale=0.35, semithick, baseline=(r11)]
  \style;
  \node [sml] (r)     at ( 0,4) {};
  \node [sml] (r1)    at ( 0,3) {};
  \node [sml] (r11)   at (-1,2) {};
  \node [sml] (r111)  at (-1,1) {};
  \node [sml] (r12)   at ( 1,2) {};
  \node [sml] (r121)  at ( 1,1) {};
  \node [sml] (r1211) at ( 1,0) {};
  \draw   (r)     node[right=1pt] {3}
       -- (r1)    node[right=1pt] {2}
       -- (r11)   node[left=1pt]  {1}
       -- (r111)  node[left=1pt]  {0}
  (r1) -- (r12)   node[right=1pt] {1}
       -- (r121)  node[right=1pt] {0}
       -- (r1211) node[right=1pt] {0};
\end{tikzpicture}
}

\def\triangle[#1,#2]#3;{
  \begin{scope}[shift={#3}]
    \begin{scope}[rotate=#2]
      \path (0,0)     coordinate (O);
      \path (245:4.4) coordinate (A);
      \path (-65:4.4) coordinate (B);
      \filldraw[gray!20] (O) -- (A) -- (B) -- cycle;
      \draw (O) -- (A);
      \draw (O) -- (B);
      \node at (0,-2.6) {#1};
      \node[o] at (O) {};
    \end{scope}
  \end{scope}
}

\renewcommand{\root}{\operatorname{root}}
\DeclareMathOperator{\exc}{\mathrm{exc}}

\DeclareMathOperator{\sub}{\mathrm{sub}}

\DeclareMathOperator{\rmod}{\mathrm{rmod}}

\DeclareMathOperator{\rzero}{\mathrm{rzero}}
\DeclareMathOperator{\open}{\mathrm{open}}

\newcommand{\one}{\operatorname{one}}
\newcommand{\rootM}{\operatorname{f\mathfrak{1}r\mathfrak{3}}}
\newcommand{\subM}{\operatorname{s\mathfrak{1}r\mathfrak{3}}}
\newcommand{\rzeroM}{\operatorname{b}}
\newcommand{\rmodM}{\operatorname{f\mathfrak{3}r\mathfrak{2}}}

\title{An involution on bicubic maps and $\beta(0,1)$-trees}

\author{Anders Claesson}
\address{
  Department of Computer and Information Sciences,
  University of Strathclyde,
  26 Richmond Street,
  Glasgow G1 1XH, United Kingdom,
  \texttt{anders.claesson@strath.ac.uk}
}
\author{Sergey Kitaev}
\address{
  Department of Computer and Information Sciences,
  University of Strathclyde,
  26 Richmond Street,
  Glasgow G1 1XH, United Kingdom,
  \texttt{sergey.kitaev@strath.ac.uk}
}
\author{Anna de Mier}
\address{
  Departament de Matem\`{a}tica Aplicada II,
  Universitat Polit\`{e}cnica de Catalunya,
  Jordi Girona 1--3,
  08034 Barcelona, Spain,
  \texttt{anna.de.mier@upc.edu}
}

\begin{document}

\begin{abstract}
  Bicubic maps are in bijection with \btrees.
  We introduce two new ways of decomposing \btrees. Using this we
  define an endofunction on \btrees, and thus also on bicubic maps. We
  show that this endofunction is in fact an involution. As a
  consequence we are able to prove some surprising results regarding
  the joint equidistribution of certain pairs of statistics on trees
  and maps. Finally, we conjecture the number of fixed points of the
  involution.
\end{abstract}


\maketitle
\thispagestyle{empty}

\section{Introduction}

A \emph{planar map} is a connected graph embedded in the sphere with
no edge-crossings, considered up to continuous deformations. A map has
\emph{vertices}, \emph{edges}, and \emph{faces} (disjoint simply
connected domains). The maps we consider shall be \emph{rooted},
meaning that a directed edge has been distinguished as the root. The
face that lies to the right of the root edge while following its
orientation is the \emph{root face}, whereas the vertex from which the
root stems is the \emph{root vertex}. When drawing a planar map on the
plane, we usually follow the convention to choose the outer
(unbounded) face as the root face. Tutte \cite[Chapter 10]{Tutte1998}
founded the enumerative theory of planar maps in a series of papers in
the 1960s (see~\cite{Tutte1963} and the references in~\cite{CS1997}).

A planar map in which each vertex is of degree 3 is {\em cubic}; it is
{\em bicubic} if, in addition, it is bipartite, that is, if its
vertices can be colored using two colors, say, black and white, so
that adjacent vertices are assigned different colors.

The smallest bicubic map has two vertices and
three edges joining them. It is well-known that the faces of a bicubic map
can be colored using three colors so that adjacent faces have distinct
colors, say, colors 1, 2 and 3, in a counterclockwise order around
white vertices. We will assume that the root vertex is black and the
root face has color 3. There are exactly three different bicubic maps
with 6 edges and they are given in Figure~\ref{bicubic}.

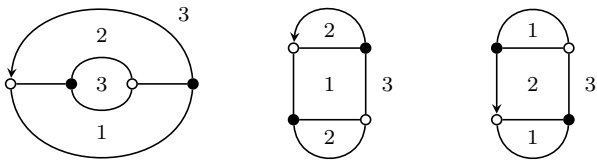
\begin{figure}[h]
  \begin{tikzpicture}[semithick, scale=0.80, baseline=(c.base), bend angle=85, >=stealth]
    \style;
    \node [wht] (a) at (0,0) {};
    \node [blk] (b) at (1,0) {};
    \node [wht] (c) at (2,0) {};
    \node [blk] (d) at (3,0) {};
    \node (n1) at (1.5,-0.8) {1};
    \node (n2) at (1.5, 0  ) {3};
    \node (n3) at (1.5, 0.8) {2};
    \node (n4) at (2.85, 1.15) {3};
    \path (a) edge[bend right, looseness=1.3] (d);
    \path (d) edge[bend right, looseness=1.3, shorten >=0.15pt, ->] (a);
    \path (b) edge[bend right, looseness=1.2] (c);
    \path (c) edge[bend right, looseness=1.2] (b);
    \path (a) edge (b) (c) edge (d);
  \end{tikzpicture}
  \qquad\quad
  \begin{tikzpicture}[semithick, scale=0.95, baseline=(n2), bend angle=85, >=stealth]
    \style;
    \node [blk] (a) at (0,0) {};
    \node [wht] (b) at (1,0) {};
    \node [wht] (c) at (0,1) {};
    \node [blk] (d) at (1,1) {};
    \node (n1) at (0.5,-0.25) {2};
    \node (n2) at (0.5, 0.50) {1};
    \node (n3) at (0.5, 1.25) {2};
    \node (n4) at (1.3, 0.50) {3};
    \path (a) edge (b) (b) edge (d) (d) edge (c) (c) edge (a);
    \path (a) edge[bend right, looseness=1.6] (b);
    \path (d) edge[bend right, looseness=1.6, shorten >=0.15pt, ->] (c);
  \end{tikzpicture}
  \qquad\quad
  \begin{tikzpicture}[semithick, scale=0.95, baseline=(n2), bend angle=85, >=stealth]
    \style;
    \node [wht] (a) at (0,0) {};
    \node [blk] (b) at (1,0) {};
    \node [blk] (c) at (0,1) {};
    \node [wht] (d) at (1,1) {};
    \node (n1) at (0.5,-0.25) {1};
    \node (n2) at (0.5, 0.50) {2};
    \node (n3) at (0.5, 1.25) {1};
    \node (n4) at (1.3, 0.50) {3};
    \path (a) edge (b) (b) edge (d) (d) edge (c) (c) edge[shorten >=0.15pt, ->] (a);
    \path (a) edge[bend right, looseness=1.6] (b);
    \path (d) edge[bend right, looseness=1.6] (c);
  \end{tikzpicture}
  \caption{All bicubic maps with 4 vertices.}\label{bicubic}
\end{figure}

The number of bicubic maps with $2n$ vertices was given by
Tutte \cite{Tutte1963}: $$\frac{3\cdot 2^{n-1}(2n)!}{n!(n+2)!}.$$

Let $M$ be a bicubic map. For $i=1,2,3$, let $\F_i(M)$ be the
set of $i$-colored faces of $M$. Let $R_1\in\F_1(M)$,
$R_2\in\F_2(M)$, and $R_3\in\F_3(M)$ be the three faces around the
root vertex; in particular, $R_3$ is the root
face.  We shall now define two statistics on bicubic maps:
\begin{align*}
\rootM(M) &\;\;\text{is the number of faces in $\F_1(M)$ that touch $R_3$;}\\
\rmodM(M) &\;\;\text{is the number of faces in $\F_3(M)$ that touch $R_2$.}
\end{align*}

Consider the following transformation $\phi$ on bicubic maps. Recolor
the faces by the mapping $\{1\mapsto 2,\, 2\mapsto 3,\, 3\mapsto 1\}$. Keep the
colors of the vertices. Keep, also, the root vertex, but let the new root edge be
the first edge in counterclockwise
direction from the old root edge:\vspace{-2ex}
$$
\begin{tikzpicture}[scale=1.4, semithick, bend angle=85, >=stealth]
  \style;
  \node [wht] (a) at (-0.9,0) {};
  \node [blk] (b) at (0,0)  {};
  \node [wht] (c) at (0:1)  {};
  \node [wht] (d) at (50:1) {};
  \path (a) edge (b);
  \path (b) edge[shorten >=0.15pt, ->]
    node[midway, yshift=-8pt] {old root} (c);
  \path (b) edge[shorten >=0.15pt, ->]
    node[midway, yshift=6pt, xshift=-7pt, rotate=50] {new root} (d);
\end{tikzpicture}
$$
It is easy to see that $\phi$ is a bijection; indeed, $\phi^3$ is the
identity transformation. Moreover, $\phi$ establishes the following
result.

\begin{proposition}\label{prop}
  For any positive integer $n$, we have
  $$\sum_{M}x^{\rootM(M)} = \sum_{M}x^{\rmodM(M)},
  $$
  where both sums are over all bicubic maps on $n$ vertices. In
other words, the statistics $\rootM$ and $\rmodM$ are equidistributed.
\end{proposition}

In this paper we show the following stronger result.

\begin{theorem}\label{thm}
  For any positive integer $n$, we have
  $$\sum_{M}x^{\rootM(M)}y^{\rmodM(M)} = \sum_{M}x^{\rmodM(M)}y^{\rootM(M)},
  $$
  where both sums are over all bicubic maps on $n$ vertices. In
  other words, the two pairs of statistics $(\rootM,
  \rmodM)$ and $(\rmodM, \rootM)$ are jointly equidistributed.
\end{theorem}

To prove Theorem~\ref{thm} we first translate the statement to a
corresponding statement on so called \btrees; there is a one-to-one
correspondence~\cite{CJS} between bicubic maps and such trees. We then
provide two proofs of the theorem.  Our first proof of Theorem
\ref{thm} is based on generating functions (see the end
of Section \ref{sec:new}).  Our combinatorial proof of the theorem
(see Corollary \ref{corollary11} and the text following it) is based
on defining an endofunction on the trees, and proving that it is an
involution that respects the statistics in question (see Theorem
\ref{thm:g-is-an-involution}). We also conjecture the number of fixed
points of the involution. 
%The proof that the endofunction is an
%involution is the most difficult part of this paper.

The results in this paper can be seen as an extension to $\beta(0,1)$-trees and bicubic maps of studies
conducted in \cite{CKS,CKS0,KM,KMN} on $\beta(1,0)$-trees and rooted
non-separable planar maps.

\section{\btrees}\label{sec:btrees}

Cori et al.~\cite{CJS} introduced description trees to give a
framework for recursively decomposing several families of planar
maps. A \btree\ is a particular kind of description tree; it is
defined as a rooted plane tree whose nodes are labeled with
nonnegative integers such that
\begin{enumerate}
\item leaves have label $0$;
\item the label of the root is one more than the sum of its children's
  labels;
\item the label of any other node exceeds the sum of its children's
  labels by at most one.
\end{enumerate}
The unique \btree\ with exactly one node (and no edges) will be called
{\em trivial}; the root of the trivial tree has label  $0$. Any other \btree\ will be called {\em nontrivial}. In
Figure~\ref{beta01} we have listed all \btrees\ on 4 nodes.
\begin{figure}[h]
  $$
  \trea{0}{0}{0}{1}\qquad
  \trea{0}{0}{1}{2}\qquad
  \trea{0}{1}{0}{1}\qquad
  \trea{0}{1}{1}{2}\qquad
  \trea{0}{1}{2}{3}\qquad
  \treb{0}{0}{0}{1}\qquad
  \treb{0}{0}{1}{2}
  $$
  $$
  \trec{0}{0}{0}{1}\qquad
  \trec{0}{1}{0}{2}\qquad
  \tred{0}{0}{0}{1}\qquad
  \tred{0}{1}{0}{2}\qquad
  \tree{0}{0}{0}{1}
  $$
  \caption{All \btrees\ on 4 nodes.}\label{beta01}
\end{figure}
Let $\root(T)$ denote the root label of $T$, and let $\sub(T)$ denote
the number of children of the root.  We say that a \btree\ $T$ is {\em
  reducible} if $\sub(T)>1$, and {\em irreducible} otherwise. Any
reducible tree can be written as a sum of irreducible ones, where the
sum $U\oplus V$ of two trees $U$ and $V$ is defined as the tree
obtained by identifying the roots of $U$ and $V$ into a new root with
label $\root(U)+\root(V)-1$.  See Figure~\ref{fig:tree} for an
example.
\begin{figure}[h]
  \figtree
  \!\!\!\!=\;\;\;\;
  \begin{tikzpicture}[xscale=0.37, yscale=0.5, semithick, baseline=(r)]
    \style;
    \node [sml] (r)       at (0,1) {};
    \node [sml] (r1)      at (0,0) {};
    \draw (r) node[above=2pt] {1} -- (r1) node[below=2pt] {0};
  \end{tikzpicture}
  \,$\oplus$\,
  \begin{tikzpicture}[xscale=0.37, yscale=0.5, semithick, baseline=(r)]
    \style;
    \node [sml] (r)       at (0,1) {};
    \node [sml] (r1)      at (0,0) {};
    \draw (r) node[above=2pt] {1} -- (r1) node[below=2pt] {0};
  \end{tikzpicture}
  \,$\oplus$\,
  \begin{tikzpicture}[xscale=0.37, yscale=0.5, semithick, baseline=(r)]
    \style;
    \node [sml] (r)       at (0,2) {};
    \node [sml] (r3)      at (0,1) {};
    \node [sml] (r31)     at (0,0) {};
    \draw (r) node[above=2pt] {2} -- (r3) node[right=1pt] {1} -- (r31) node[below=2pt] {0};
  \end{tikzpicture}
  $\!\oplus\!$\hspace{-2.5ex}
  \begin{tikzpicture}[xscale=0.37, yscale=0.5, semithick, baseline=(r)]
    \style;
    \node [sml] (r)       at (3,6) {};
    \node [sml] (r4)      at (3,5) {};
    \node [sml] (r41)     at (3,4) {};
    \node [sml] (r411)    at (3,3) {};
    \node [sml] (r4111)   at (2,2) {};
    \node [sml] (r41111)  at (2,1) {};
    \node [sml] (r411111) at (2,0) {};
    \node [sml] (r4112)   at (4,2) {};
    \draw (r) node[above=2pt] {3}
          (r)                     -- (r4)  node[right=2pt] {2}
          (r4)                    -- (r41) node[right=2pt] {1}
          (r41)                   -- (r411) node[right=2pt] {3}
          (r411)                  -- (r4111) node[left=2pt] {2}
          (r4111)                 -- (r41111) node[left=2pt] {1}
          (r41111)                -- (r411111) node[left=2pt] {0}
          (r411)                  -- (r4112) node[right=2pt] {0};
  \end{tikzpicture}
  \caption{Decomposing a reducible \btree.}\label{fig:tree}
\end{figure}

Note also that any irreducible tree with at least one edge
is of the form $\lambda_i(T)$, where $0 \leq i\leq \root(T)$ and
$\lambda_i(T)$ is obtained from $T$ by joining a new root via an
edge to the old root; the old root is given the label $i$, and the new
root is given the label $i+1$. For instance,
$$
\text{if }\;
T =
\begin{tikzpicture}[xscale=0.37, yscale=0.5, semithick, baseline=(r)]
  \style;
  \node [sml] (r)  at ( 0,2) {};
  \node [sml] (r1) at (-1,1) {};
  \node [sml] (r2) at ( 1,1) {};
  \node [sml] (r21) at ( 1,0) {};
  \draw (r) node[above=2pt] {2} -- (r1) node[below=2pt] {0}
        (r)                     -- (r2) node[right=1pt] {1} -- (r21) node[below=2pt] {0};
\end{tikzpicture}
\text{then}\quad
\newcommand{\lambdaT}[2]{
\begin{tikzpicture}[xscale=0.37, yscale=0.5, semithick, baseline=(r')]
  \style;
  \node [sml] (r') at ( 0,3) {};
  \node [sml] (r)  at ( 0,2) {};
  \node [sml] (r1) at (-1,1) {};
  \node [sml] (r2) at ( 1,1) {};
  \node [sml] (r21) at ( 1,0) {};
  \draw (r') node[above=2pt] {#2} -- (r);
  \draw (r)  node[right=2pt] {#1} -- (r1) node[below=2pt] {0}
        (r)                       -- (r2) node[right=1pt] {1} -- (r21) node[below=2pt] {0};
\end{tikzpicture}
}
\lambda_0(T) = \!\!\!\!\lambdaT{0}{1}\hspace{-3.5ex},\quad
\lambda_1(T) = \!\!\!\!\lambdaT{1}{2}\hspace{-3.5ex},\quad \text{and }\;
\lambda_2(T) = \!\!\!\!\lambdaT{2}{3}\hspace{-3.5ex}.\quad
$$

Let us now introduce a few more statistics on \btrees. By the {\em
  rightmost path} we shall mean the path from the root to the
rightmost leaf. We define $\rzero(T)$ as the number of zeros on the
rightmost path. By definition, $\rzero(\leaf)=0$.

A node is called {\em excessive} if its label exceeds the sum of its
children's labels; it is called {\em moderate} otherwise. In
particular, a leaf is a moderate node and the root is an excessive
node.  Assuming that $T$ is nontrivial, we let $\rmod(T)$ be the
number of moderate nodes on the rightmost path of $T$. For the case of
the trivial tree we define $\rmod(\leaf)=0$.

A node on the rightmost path, possibly the root, will be called {\em
  open} if its rightmost child (the child on the rightmost path), if
any, is a non-leaf moderate node. In particular, the rightmost
leaf is always an open node. Let $\open(T)$ denote the number of
open nodes in $T$; we define $\open(\leaf)=0$.

For the tree $T$ in Figure~\ref{fig:tree} we see that $\root(T)=4$,
$\sub(T)=4$, $\rzero(T)=1$ and $\rmod(T)=\open(T)=2$. That $\rmod(T)$ and
$\open(T)$ agree is not a coincidence as demonstrated in the proof of
the following lemma.

\begin{lemma}\label{lemma:rmod-open}
  For any \btree\ $T$ we have $\rmod(T)=\open(T)$.
\end{lemma}

\begin{proof}
  Since the right child of a non-leaf open node is non-leaf and
  moderate, and the root is not a moderate node, it follows that among
  non-leaves the numbers of open and moderate nodes agree. As the
  rightmost leaf is both open and moderate, the equality of both
  statistics follows.
%  Let $\rpath(T)$ be the number of nodes on the rightmost path, and
%  let $\rex(T)$ be the number of excessive nodes below the root on the
%  rightmost path. We have $\open(T) = \rpath(T) -\rex(T)
%  -1$, where the $-1$ accounts for the parent of the rightmost
%  leaf. We also have $\rpath(T) = \rmod(T) + \rex(T) + 1$, where the
%  $+1$ accounts for the root. The claim follows.
\end{proof}

\section{Bicubic maps as \btrees}\label{sec:maps-trees}

Following \cite{CS1997} we will now describe a bijection between 
bicubic maps and \btrees. Let us first recall some definitions from the
introduction. For any bicubic map $M$ and $i=1,2,3$, let
$\F_i(M)$ be the set of $i$-colored faces of $M$. Let
$R_1\in\F_1(M)$, $R_2\in\F_2(M)$, and $R_3\in\F_3(M)$ be the three
faces around the root vertex; in particular, $R_3$
is the root face. In addition, let $S_1\in\F_1(M)$ be the $1$-colored
face that meets the vertex that the root edge points at:
$$
\begin{tikzpicture}[semithick, scale=0.7, baseline=-0.6ex]
  \style;
  \node [blk] (r) at (0,0) {};
  \node [wht] (s) at (1,0) {};
  \path (-1,0) edge (r) edge[->] (s) edge (2,0);
  \path (r) edge (0,1);
  \path (s) edge (1,1);
  \node [wht] (s) at (1,0) {};
  \node at (-0.5, 0.5) {$R_1$};
  \node at ( 0.5, 0.5) {$R_2$};
  \node at ( 1.5, 0.5) {$S_1$};
  \node at ( 0.5,-0.5) {$R_3$};
\end{tikzpicture}
$$ 
Let us say that a face touches another face $k$ times if there are $k$
different edges each belonging to the boundaries of both faces.
Define the following two statistics:
\begin{align*}
  \rzeroM(M)&\;\;\text{is the number of black vertices incident to both $R_1$ and $R_2$;}\\
  \subM(M)  &\;\;\text{is the number of times $S_1$ touches $R_3$.}
\end{align*}
We say that $M$ is {\em irreducible} if $\subM(M)=1$, or, in other
words, if $S_1$ touches $R_3$ exactly once; we say that $M$ is {\em
  reducible} otherwise. We shall introduce operations on bicubic
maps that correspond to $\lambda_i$ and $\oplus$ of \btrees.  This
will induce the desired bijection $\psi$ between bicubic maps and
$\beta(0,1)$-trees.

To construct an irreducible bicubic map based on $M$, and having two
more vertices than $M$, we proceed in one of two ways. The first way (1)
corresponds to  $\lambda_i(T)$ when $i=\root(T)$; the second way (2)
corresponds to $\lambda_i(T)$ when $0\leq i < \root(T)$.

\def\mapM{ \filldraw[fill=blue!6,draw=black!50] (0,0) circle (1);
  \node[font=\small] at (0,0) {$M$}; }

\begin{itemize}
\item[(1)] We create a new $1$-colored face touching the root face
  exactly once, so $\rootM(M')=\rootM(M)+1$, by removing the root edge
  from $M$ and adding a digon that we connect to the map as in the
  picture below.
  $$
  \begin{tikzpicture}[semithick, scale=0.65, baseline=-0.6ex]
    \style;
    \mapM;
    \draw [->] (240:1) arc (240:291:1);
    \node [blk] (r) at (240:1) {};
    \node [wht] (s) at (300:1) {};
    \node at (-1.4,0) {\tiny{3}};
  \end{tikzpicture}
  \quad \longmapsto \quad M' \,=\; 
  \begin{tikzpicture}[semithick, scale=0.65, baseline=-0.6ex]
    \style;
    \mapM;
    \draw[white, thick] (240:1) arc (240:300:1);
    \draw (240:1) -- +(0,-1) coordinate (a);
    \draw[->] (240:1) -- +(0,-0.84);
    \draw (300:1) -- +(0,-1) coordinate (b);
    \path (a) edge[bend right, looseness=1.5] (b);
    \path (a) edge[bend left, looseness=1.5] (b);
    \node [wht] at (a) {};
    \node [blk] at (b) {};
    \node [blk] (r) at (240:1) {};
    \node [wht] (s) at (300:1) {};
    \node at ($ (a) !.5! (b) $) {\tiny{1}};
    \node[yshift=0.5pt] at ($ (r) !.5! (b) $) {\tiny{2}};
    \node at (1.4,0) {\tiny{3}};
  \end{tikzpicture}
  $$
\item[(2)] Assuming that $\rootM(M)=k$; that is, $M$ has $k$ $1$-colored
  faces touching the root face, we can create an irreducible map $M'$
  such that $\rootM(M')=i$, where $1\leq i\leq k$.  To this end, we
  remove the root edge from $M$. Starting at the root node and
  counting in {\em clockwise direction}, we also remove the first edge
  of the $i$-th $1$-colored face that touches the root face. In the
  picture below we schematically illustrate the case $i=3$. Next we add
  two more vertices and respective edges, and assign a new root as shown in the
  figure.
  \def\radius{1.5}
  \def\radiux{1.3}
  $$ 
  \begin{tikzpicture}[semithick, scale=0.65, bend angle=45, rotate=0,  baseline=-0.6ex] % -22.5
    \style;
    \path
    coordinate (P0) at (0*45:\radius) {}
    coordinate (P1) at (1*45:\radius) {}
    coordinate (P2) at (2*45:\radius) {}
    coordinate (P3) at (3*45:\radius) {}
    coordinate (P4) at (4*45:\radius) {}
    coordinate (P5) at (5*45:\radius) {}
    coordinate (P6) at (6*45:\radius) {}
    coordinate (P7) at (7*45:\radius) {};
    \filldraw[fill=blue!6,draw=black!50] (0,0) circle (\radius);
    \node[font=\small] at (0,0) {$M$};
    \path (P1) edge[bend left, looseness=1.5] (P2);
    \path (P3) edge[bend left, looseness=1.5] (P4);
    \path (P5) edge[bend left, looseness=1.5] (P6);
    \path (P7) edge[bend left, looseness=1.5] (P0);
    \draw [->] (P6) arc (6*45:7*45-5:\radius);
    \node at (-2,0) {\tiny{3}};
    \node at (1.5*45:\radiux) {\tiny{1}};
    \node at (3.5*45:\radiux) {\tiny{1}};
    \node at (5.5*45:\radiux) {\tiny{1}};
    \node at (7.5*45:\radiux) {\tiny{1}};
    \node [blk] at (P0) {};
    \node [wht] at (P1) {};
    \node [blk] at (P2) {};
    \node [wht] at (P3) {};
    \node [blk] at (P4) {};
    \node [wht] at (P5) {};
    \node [blk] at (P6) {};
    \node [wht] at (P7) {};
  \end{tikzpicture}
  \quad \longmapsto \quad M' \,=\!\!
  \begin{tikzpicture}[semithick, scale=0.65, bend angle=45, rotate=0,  baseline=-0.6ex] % -22.5
    \style;
    \path
    coordinate (A)  at (0*45:2*\radius) {}
    coordinate (B)  at (0*45:2.5*\radius) {}
    coordinate (Q)  at (1.7*45:\radius) {}
    coordinate (P0) at (0*45:\radius) {}
    coordinate (P1) at (1*45:\radius) {}
    coordinate (P2) at (2*45:\radius) {}
    coordinate (P3) at (3*45:\radius) {}
    coordinate (P4) at (4*45:\radius) {}
    coordinate (P5) at (5*45:\radius) {}
    coordinate (P6) at (6*45:\radius) {}
    coordinate (P7) at (7*45:\radius) {};
    \filldraw[fill=blue!6,draw=black!50] (0,0) circle (\radius);
    \node[font=\small] at (0,0) {$M$};
    \path (P6) edge[->, bend right, shorten >=2pt] (B);
    \path (P7) edge[bend right] (A);
    \path (B) edge[bend angle=90, bend right, looseness=1] (P2);
    \path (A) edge[bend angle=90, bend right, looseness=1] (Q);
    \path (A) edge (B);
    \path (P1) edge[bend left, looseness=1.5] (P2);
    \path (P3) edge[bend left, looseness=1.5] (P4);
    \path (P5) edge[bend left, looseness=1.5] (P6);
    \path (P7) edge[bend left, looseness=1.5] (P0);
    \draw[very thick, white] (Q) arc (1.7*45:2*45:\radius);
    \draw[very thick, white] (P6) arc (6*45:7*45:\radius);
    \node at (4.5*45:1.9) {\tiny{3}};
    \node at (0.5*45:1.9) {\tiny{3}};
    \node at (7.0*44.3:1.9) {\tiny{2}};
    \node at (0.5*45:3.17) {\tiny{1}};
    \node at (3.5*45:\radiux) {\tiny{1}};
    \node at (5.5*45:\radiux) {\tiny{1}};
    \node at (7.5*45:\radiux) {\tiny{1}};
    \node [blk] at (A)  {};
    \node [wht] at (B)  {};
    \node [blk] at (P0) {};
    \node [wht] at (P1) {};
    \node [blk] at (P2) {};
    \node [wht] at (P3) {};
    \node [blk] at (P4) {};
    \node [wht] at (P5) {};
    \node [blk] at (P6) {};
    \node [wht] at (P7) {};
    \node [circle,fill=white,draw=black, minimum size=2.5pt, inner sep=0pt] at (Q) {};
  \end{tikzpicture}
  $$
\end{itemize}
Any irreducible bicubic map on $n+2$ vertices can be constructed from some
bicubic map on $n$ vertices by applying operation (1) or (2) above.

We shall now describe how to create a reducible map based on
irreducible maps $M_1$, $M_2$, \ldots, $M_k$. An illustration for
$k=3$ can be found below. This corresponds to the $\oplus$-operation on \btrees.
\begin{itemize}
\item[(3)] We begin by lining up the maps $M_1$, $M_2$, \ldots,
  $M_k$. Next, in each map $M_i$, we remove the first edge (in {\em
    counter-clockwise direction}) from the root edge on the root
  face. Then we connect the maps as shown in the figure, and define
  the root edge of the obtained map to be the root edge of $M_k$.
\end{itemize}

\def\mapMi{\filldraw[fill=blue!6,draw=black!50] (0,0) circle (1);}
$$
\begin{tikzpicture}[semithick, scale=0.61, bend angle=45, rotate=0,  baseline=-0.6ex] %
  \style;
  \mapMi
  \draw [->] (180:1) arc (180:231:1);
  \draw      (240:1) arc (240:300:1);
  \node [blk] (r1) at (180:1) {};
  \node [wht] (s1) at (240:1) {};
  \node [blk] (t1) at (300:1) {};
  \path (s1) edge[bend left, looseness=1.5] (t1);
  \node [yshift=1.3] at ($ (s1) !.5! (t1) $) {\tiny{1}};
  \node[font=\small] at (0,0) {$M_1$};
  \node at (-1.2,0.7) {\tiny{3}};
  \begin{scope}[xshift=80pt]
    \mapMi
    \draw [->] (180:1) arc (180:231:1);
    \draw      (240:1) arc (240:300:1);
    \node [blk] (r2) at (180:1) {};
    \node [wht] (s2) at (240:1) {};
    \node [blk] (t2) at (300:1) {};
    \path (s2) edge[bend left, looseness=1.5] (t2);
    \node [yshift=1.3] at ($ (s2) !.5! (t2) $) {\tiny{1}};
    \node[font=\small] at (0,0) {$M_2$};
    \node at (-1.2,0.7) {\tiny{3}};
  \end{scope}
  \begin{scope}[xshift=160pt]
    \mapMi
    \draw [->] (180:1) arc (180:231:1);
    \draw      (240:1) arc (240:300:1);
    \node [blk] (r3) at (180:1) {};
    \node [wht] (s3) at (240:1) {};
    \node [blk] (t3) at (300:1) {};
    \path (s3) edge[bend left, looseness=1.5] (t3);
    \node [yshift=1.3] at ($ (s3) !.5! (t3) $) {\tiny{1}};
    \node[font=\small] at (0,0) {$M_3$};
    \node at (-1.2,0.7) {\tiny{3}};
  \end{scope}
\end{tikzpicture}
\quad \longmapsto \quad M' \,=\;
\begin{tikzpicture}[semithick, scale=0.61, bend angle=45, rotate=0,  baseline=-0.6ex] %
  \style;
  \mapMi
  \draw[->]           (180:1) arc (180:231:1);
  \draw[thick, white] (240:1) arc (240:300:1);
  \node [blk] (r1) at (180:1) {};
  \node [wht] (s1) at (240:1) {};
  \node [blk] (t1) at (300:1) {};
  \path (s1) edge[bend left, looseness=1.5] (t1);
  \node[font=\small] at (0,0) {$M_3$};
%  \node at (-1.2,0.7) {\tiny{3}};
  \begin{scope}[xshift=80pt]
    \mapMi
    \draw               (180:1) arc (180:231:1);
    \draw[thick, white] (240:1) arc (240:300:1);
    \node [blk] (r2) at (180:1) {};
    \node [wht] (s2) at (240:1) {};
    \node [blk] (t2) at (300:1) {};
    \path (s2) edge[bend left, looseness=1.5] (t2);
    \node[font=\small] at (0,0) {$M_2$};
    \node at (0,-1.4) {\tiny{1}};
  \end{scope}
  \begin{scope}[xshift=160pt]
    \mapMi
    \draw               (180:1) arc (180:231:1);
    \draw[thick, white] (240:1) arc (240:300:1);
    \node [blk] (r3) at (180:1) {};
    \node [wht] (s3) at (240:1) {};
    \node [blk] (t3) at (300:1) {};
    \path (s3) edge[bend left, looseness=1.5] (t3);
    \node[font=\small] at (0,0) {$M_1$};
  \end{scope}
  \path (t1) edge[bend right=30] (s2)
        (t2) edge[bend right=30] (s3)
        (s1) edge[bend right=50, looseness=0.6] (t3);
  \node at (1.35,0.8) {\tiny{3}};
\end{tikzpicture}
$$
Any reducible bicubic map on $n$ vertices can be constructed by applying
the above operation (3) to some ordered list of irreducible bicubic maps whose
total number of vertices is $n$.

By defining operations on bicubic maps corresponding to the operations
$\lambda_i$ and $\oplus$ we have now completed the definition of the
bijection $\psi$ between bicubic maps and \btrees. Two examples of
applying $\psi$ can be found in the appendix.

\begin{proposition}\label{prop:exc}
  Let $M$ be a bicubic map, and let $\one(M) = |\F_1(M)|$ be the
  number of $1$-colored faces in $M$.  Let $T$ be a \btree, and let
  $\exc(T)$ denote the number of excessive nodes in $T$.  Let $\psi$
  be the map from bicubic maps to \btrees\ described above. Finally,
  assume that $T = \psi(M)$. Then
  \begin{align*}
    \exc(T)   &= \one(M); \\
    \root(T)  &= \rootM(M); \\
    \rmod(T)  &= \rmodM(M); \\
    \rzero(T) &= \rzeroM(M); \\
    \sub(T)   &= \subM(M).
  \end{align*}
\end{proposition}

\newcommand{\edge}{
  \begin{tikzpicture}[scale=0.4, semithick, baseline=3pt]
    \style;
    \node [sml] (r)       at (0,1) {};
    \node [sml] (r1)      at (0,0) {};
    \draw (r) node[left=1pt]  {\tiny{1}} -- (r1) node[left=1pt] {\tiny{0}};
  \end{tikzpicture}
}
\newcommand{\edgeM}{
  \begin{tikzpicture}[semithick, scale=0.7, baseline=-3pt]
    \style;
    \node [blk] (r) at (0,0) {};
    \node [wht] (s) at (1,0) {};
    \path     (r) edge                               (s);
    \path     (r) edge[bend left=67 , looseness=1.5] (s);
    \path[->] (r) edge[bend right=67, looseness=1.5] (s);
    \node at (-0.1, 0.31) {\tiny{3}};
    \node at ( 0.5, 0.21) {\tiny{1}};
    \node at ( 0.5,-0.21) {\tiny{2}};
  \end{tikzpicture}
}
\begin{proof}
  The proofs of these six equalities are similar, and we will only
  detail the proof of $\rzero(T) = \rzeroM(M)$; the proofs of the
  other equalities are simpler. Clearly,
  $$\rzero\left(\!\edge\,\right) = \rzeroM\left(\!\!\edgeM\right) = 1.
  $$
  Let $M'$ be a bicubic map with at least $4$ vertices. Then $M'$ can
  be constructed from one ($M$) or more ($M_1,\dots,M_k$) smaller
  bicubic maps as per the three rules above.
  \begin{itemize}
  \item[(1)] Assume that $T$ and $T'$ are the trees corresponding to
    $M$ and $M'$, respectively. Then $T' = \lambda_i(T)$ with
    $i=\root(T)$. The labels on the rightmost path of $T$ are
    preserved in $T'$, and a new nonzero (root) node is added. Thus
    $\rzero(T') = \rzero(T)$. We need to show that $\rzero(M') =
    \rzero(M)$, but this easy to see from the picture above: the only
    black vertex added is not incident to $R_1$, and the status (incident or not incident to $R_1$ and $R_2$) of each
    of the other black vertices incident to both $R_1$ and $R_2$ is
    preserved.
  \item[(2)] Here $T' = \lambda_i(T)$ with $0\leq i < \root(T)$, and
    we distinguish two sub-cases.
    \begin{itemize}
    \item[(a)] Assume that $i=0$. Comparing $T$ to $T'$ we see that
      one more zero appears on the rightmost path of $T'$, namely the
      new root. Thus $\rzero(T')=\rzero(T)+1$. On the map $M'$ we have
      just one 1-colored face touching $R_3$ and this face must be
      $R_1$. Comparing $M$ to $M'$ we see that the black vertex added
      to $M$ in order to form $M'$ is incident to both $R_1$ and
      $R_2$. The status 
      of each of the other black vertices is preserved. Thus
      $\rzeroM(M')=\rzeroM(M)+1$.
    \item[(b)] Assume that $i>0$. Clearly, $\rzero(T')=\rzero(T)$. The
      black vertex added to $M$ in order to form $M'$ is not incident
      with $R_1$, and the status of each of the other black vertices
      is preserved. Thus $\rzeroM(M')=\rzeroM(M)$.
    \end{itemize}
  \item[(3)] Assume that $T_1,\dots,T_k$ and $T'$ are the trees
    corresponding to $M_1,\dots,M_k$ and $M'$, respectively. Clearly,
    $\rzero(T')=\rzero(T_k)$. Consider $M'$: no black vertex in
    $M_1,\dots,M_{k-1}$ can contribute to the $\rzeroM$-statistic
    because such a vertex is neither incident to $R_1$ nor incident to
    $R_2$. Since the status of each of the black vertices in $M_k$ is
    preserved it follows that $\rzeroM(M') = \rzeroM(M_k)$.
  \end{itemize}
  The result now follows by induction.
\end{proof}

\section{New ways to decompose \btrees}\label{sec:new}

For any \btrees\ $T_1$, $T_2$, \dots, $T_k$ define 
$$\rho(T_1,T_2,\dots,T_k)
= \lambda_0(T_1)\oplus\lambda_0(T_2)\oplus\dots\oplus\lambda_0(T_k).
$$
Let $S$ and $T$ be \btrees. Assume that $\root(S)=1$ and that $T$ is
nontrivial. Let $i$ be an integer such that $1\leq i\leq \open(T)$,
and let $x$ denote the $i$th open node on the rightmost path of
$T$. Also, let $y$ be $x$ if $x$ is a leaf and let $y$ be the
rightmost child of $x$ otherwise. We define $\mu_i(S,T)$ as the
\btree\ obtained by identifying $x$ with the root of $S$, keeping the
label of $x$, and then adding one to each node on the rightmost path
of $T$ between the root and $y$. For instance,
$$
\mu_2\left(
\treeS,\;
\treeT
\right) \quad=\quad
\begin{tikzpicture}[yscale=0.5, xscale=0.42, semithick, baseline=(r11)]
  \style;
  \node [blk] (r)     at ( 0,4) {};
  \node [blk] (r1)    at ( 0,3) {};
  \node [sml] (r11)   at (-1.4,2) {};
  \node [sml] (r111)  at (-1.4,1) {};
  \node [blk] (r12)   at ( 1.4,2) {};
  \node [blk] (r121)  at ( 0,1) {};
  \node [wht] (r122)  at ( 1.4,1) {};
  \node [wht] (r1221) at ( 1.4,0) {};
  \node [wht] (r123)  at ( 2.8,1) {};
  \node [sml] (r1211) at ( 0,0) {};
   \draw
   (r)  -- (r1)    node[right=1pt, yshift=2pt] {3}
        -- (r11)   node[left=1pt]  {1}
        -- (r111)  node[left=1pt]  {0};
  \draw [semithick]
   (r) node[right=1pt] {4} --
   (r1) -- (r12)   node[right=1pt, yshift=2pt] {2}
        -- (r121)  node[left=0pt]  {1};
  \draw
  (r121) -- (r1211) node[left=0pt]  {0}
  (r12) -- (r122)  node[right=1pt] {0}
        -- (r1221) node[right=1pt] {0}
  (r12) -- (r123)  node[right=1pt] {0};
\end{tikzpicture}
$$
Schematically,
$$
\mu_i(S,T) =\;
\begin{tikzpicture}[scale=0.55, semithick, baseline=-23pt]
  \ostyle
  \begin{scope}[rotate=-30]
    \triangle[$T$,0](0,0);
    \node[o] at (-65:0.8) {};
    \node[o] at (-65:2*0.8) {};
    \triangle[$S$,70](-65:3*0.8);
    \node[o] at (-65:4*0.8) {};
    \node[xshift= 10pt, yshift=2pt] at (-65:0*0.8) {\footnotesize{$+1$}};
    \node[xshift= 10pt, yshift=2pt] at (-65:1*0.8) {\footnotesize{$+1$}};
    \node[xshift= 10pt, yshift=2pt] at (-65:2*0.8) {\footnotesize{$+1$}};
    \node[xshift= 10pt, yshift=2pt] at (-65:3*0.8) {\footnotesize{$+1$}};
    \node[xshift=-10pt, yshift=1pt] at (-65:4*0.8) {\footnotesize{$+1$}};
  \end{scope}
  \node at (6,-2) {the $i$th open node};
  \draw[>=stealth,->] (3.3,-2) -- (1.1,-2.3);
\end{tikzpicture}
$$
For convenience we shall also define that $\mu_1(S,\leaf)=S$.

Note that any \btree\ $U$ with $\root(U)=1$ is of the form
$\rho(T_1,T_2,\dots,T_k)$ for some \btrees\ $T_1$, $T_2$, \dots,
$T_k$. On the other hand, any \btree\ $U$ with $\root(U)>1$ can be
written $U = \mu_i(S,T)$, where $\root(S)=1$ and $T$ is
nontrivial. Indeed, the node we call $x$ above is the parent node of
the first node labelled $0$ on the right path of $U$, and knowing $x$
we trivially get $S$ and $T$.

Thus we can completely decompose any \btree\ in terms of
$\rho$ and $\mu_i$. As an example, the tree from Figure~\ref{fig:tree}
can be written
$$ \mu_2( \rho[\leaf]
        , \mu_1( \rho[\mu_2(\rho[\leaf],\mu_1(\rho[\rho[\leaf]],\rho[\leaf]))]
               , \mu_1(\rho[\leaf],\rho[\leaf,\leaf,\rho[\leaf]]))).
$$
% B2(A[*],B1(A[B2(A[*],B1(A[A[*]],A[*]))],B1(A[*],A[*,*,A[*]])))

We shall now define two additional operations $\sigma$ and $\nu_i$ on
\btrees\ that in a sense are dual to $\rho$ and $\mu_i$. We start with
$\sigma$: 
\begin{definition}\label{def:sigma}
  For \btrees\ $T_1$, \dots, $T_k$ define
  \begin{equation*}
    \sigma(T_1,\dots,T_{k}) = 
    \mu_1(\rho(T_{k-1},\dots,T_1,\leaf), T_k).
  \end{equation*}
\end{definition}
For example,
$$
\sigma\left[\,
\begin{tikzpicture}[yscale=0.5, xscale=0.35, semithick, baseline=(r1)]
  \style;
  \node [wht] (r)   at ( 0,3) {};
  \node [wht] (r1)  at ( 0,2) {};
  \node [wht] (r11) at (-1,1) {};
  \node [wht] (r12) at ( 1,1) {};
  \draw    (r)   node[left=1pt] {2}
        -- (r1)  node[left=1pt]  {1}
        -- (r11) node[below=1pt] {0}
   (r1) -- (r12) node[below=1pt] {0};
\end{tikzpicture}
,
\begin{tikzpicture}[yscale=0.5, xscale=0.35, semithick, baseline=(r1)]
  \style;
  \node [wht] (r)   at ( 0,3) {};
  \node [wht] (r1)  at ( 0,2) {};
  \draw    (r)   node[left=1pt] {1}
        -- (r1)  node[left=1pt, yshift=-2pt]  {0};
\end{tikzpicture}
\,,\;\,
\begin{tikzpicture}[yscale=0.5, xscale=0.35, semithick, baseline=(r1)]
  \style;
  \node [blk] (r)    at ( 0,3) {};
  \node [blk] (r1)   at ( 0,2) {};
  \node [blk] (r11)  at ( 0,1) {};
  \node [blk] (r111) at ( 0,0) {};
  \draw    (r)    node[left=1pt] {2}
        -- (r1)   node[left=1pt] {1}
        -- (r11)  node[left=1pt] {0}
        -- (r111) node[left=1pt] {0};
\end{tikzpicture}\;\;
\right] \quad=\quad
\begin{tikzpicture}[yscale=0.54, xscale=0.42, semithick, baseline=(r11)]
  \style;
  \node [blk] (r)     at (   0,4) {};
  \node [blk] (r1)    at (   0,3) {};
  \node [blk] (r11)   at (-1.4,2) {};
  \node [blk] (r111)  at (-1.4,1) {};
  \node [wht] (r12)   at (   0,2) {};
  \node [wht] (r121)  at (   0,1) {};
  \node [wht] (r13)   at ( 1.4,2) {};
  \node [wht] (r131)  at ( 1.4,1) {};
  \node [wht] (r1311) at ( 0.7,0) {};
  \node [wht] (r1312) at ( 2.1,0) {};
   \draw
    (r) node[right=1pt] {3}
         -- (r1)    node[right=1pt, yshift=2pt] {2}
         -- (r11)   node[left=1pt]  {1}
         -- (r111)  node[left=1pt]  {0}
   (r1)  -- (r12)   node[left=1pt]  {0}
         -- (r121)  node[left=1pt]  {0}
   (r1)  -- (r13)   node[right=1pt] {0}
         -- (r131)  node[right=1pt] {1}
         -- (r1311) node[left=1pt]  {0}
  (r131) -- (r1312) node[right=1pt] {0};
\end{tikzpicture}
$$

Let $S$ and $T$ be \btrees. Assume that $\open(S)=1$ and that $T$ is
nontrivial. Let $i$ be an integer such that $1\leq i\leq \root(T)$
and let $x$ denote the rightmost leaf of $S$. Define $\nu_i(S,T)$ as
the \btree\ obtained by identifying $x$ with the
root of $T$, keeping the (zero) label of $x$, and then adding $i-1$ to
each node on the rightmost path of $S$ between the root and $x$. For
instance,
$$\nu_2\left(
\begin{tikzpicture}[yscale=0.5, xscale=0.35, semithick, baseline=35pt]
  \style;
  \node [wht] (r)    at ( 0,4) {};
  \node [wht] (r1)   at ( 0,3) {};
  \node [wht] (r11)  at (-1,2) {};
  \node [wht] (r111) at (-1,1) {};
  \node [wht] (r12)  at ( 1,2) {};
  \draw    (r)    node[left=1pt, yshift=2pt] {2}
        -- (r1)   node[left=1pt, yshift=2pt] {1}
        -- (r11)  node[left=1pt]  {0}
        -- (r111) node[left=1pt]  {0}
   (r1) -- (r12)  node[below=1pt] {0};
\end{tikzpicture}
,\;
\begin{tikzpicture}[yscale=0.5, xscale=0.35, semithick, baseline=35pt]
  \style;
  \node [sml] (r)     at ( 0,4) {};
  \node [sml] (r1)    at ( 0,3) {};
  \node [sml] (r11)   at (-1,2) {};
  \node [sml] (r12)   at ( 1,2) {};
  \node [sml] (r121)  at ( 1,1) {};
  \draw   (r)     node[right=1pt, yshift=2pt] {3}
       -- (r1)    node[right=1pt, yshift=2pt] {2}
       -- (r11)   node[below=1pt] {0}
  (r1) -- (r12)   node[right=1pt] {1}
       -- (r121)  node[right=1pt] {0};
\end{tikzpicture}
\right) \quad=\quad
\begin{tikzpicture}[yscale=0.5, xscale=0.35, semithick, baseline=20pt]
  \style;
  \node [wht] (r)    at ( 0,4) {};
  \node [wht] (r1)   at ( 0,3) {};
  \node [wht] (r11)  at (-1,2) {};
  \node [wht] (r111) at (-1,1) {};
  \node [wht] (r12)  at ( 1,2) {};
  \node [sml] (s1)   at ( 1,1) {};
  \node [sml] (s11)  at ( 0,0) {};
  \node [sml] (s12)  at ( 2,0) {};
  \node [sml] (s121) at ( 2,-1) {};
  \draw    (r)    node[left=1pt, yshift=2pt] {3}
        -- (r1)   node[left=1pt, yshift=2pt] {2}
        -- (r11)  node[left=1pt]  {0}
        -- (r111) node[left=1pt]  {0}
   (r1) -- (r12)  node[right=1pt, yshift=2pt] {1};
  \draw    (r12)
        -- (s1)    node[right=1pt, yshift=2pt] {2}
        -- (s11)   node[below=1pt] {0}
   (s1) -- (s12)   node[right=1pt] {1}
        -- (s121)  node[right=1pt] {0};
\end{tikzpicture}
$$
For convenience we shall also define that $\nu_1(S,\leaf)=S$.

Note that any \btree\ $U$ with $\open(U)=1$ is of the form
$\sigma(T_1,T_2,\dots,T_k)$ for some \btrees\ $T_1$, $T_2$, \dots,
$T_k$, and any \btree\ $U$ with $\open(U)>1$ can be written $U =
\nu_i(S,T)$, where $\open(S)=1$ and $T$ is nontrivial. Again, using
the tree from Figure~\ref{fig:tree} as an example we have
$$ \nu_2( \sigma[\sigma[\nu_1(\sigma[\leaf,\leaf,\leaf],\sigma[\leaf])]]
        , \sigma[\nu_2(\sigma[\sigma[\leaf]],\sigma[\sigma[\leaf]])]).
$$
% B2(A[A[B1(A[*,*,*],A[*])]],A[B2(A[A[*]],A[A[*]])])

The behaviour of the statistics $\root$ and $\open$ under $\rho, \mu_i, \sigma$ and $\nu_i$ follows easily from the definitions.

\begin{lemma}\label{lem:behaviour}
If $T_1$, \dots, $T_k$, $S$ and $T$
  are \btrees, then
  \begin{align}
    \open(\rho(T_1,\dots,T_k)) &= 1+\open(T_k),    \label{1b} \\
    \root(\mu_i(S,T))&=1+\root(T), \label{rm} \\ 
     \open(\mu_i(S,T)) &= i - 1 + \open(S),         \label{2d} \\   
 \root(\sigma(T_1,\dots,T_k)) &= 1+\root(T_k),  \label{2a} \\
\root(\nu_i(S,T)) &= i - 1 + \root(S),         \label{1c} \\
   \open(\nu_i(S,T))&= 1+ \open(T), \label{on}
      \end{align}
  where in \eqref{rm} and \eqref{2d} we
  assume that $\root(S)=1$, and in \eqref{1c} and \eqref{on} we assume that $\open(S)=1$.
\end{lemma}

We can now give a generating function proof of Theorem \ref{thm}. \\

\begin{proof}[Proof of Theorem \ref{thm}]
  Let $F(x,y) := F(t,x,y)$ be the generating function for \btrees\ where
  $t$ marks the number of edges, $x$ marks the $\root$ statistic, and
  $y$ marks the $\rmod$ statistic.  We claim that 
  $$ F(x,y) = 1 + xS + \frac{x}{y-1} S\bigl(F(x,y)-F(x,1)\bigr),
  $$
  where $S := tyF(1,y) / (1-tF(1,1))$.
Indeed, the second and third summands correspond to  \btrees\ with root label equal to or greater than $1$, respectively, generated with the $(\rho,\mu_i)$
  decomposition. (Recall that the statistics $\rmod$ and $\open$ coincide.)

  Let now $G(x,y) := G(t,x,y)$ be the generating function for \btrees\
  where $t$ marks the number of edges, $x$ marks the $\rmod$
  statistic, and $y$ marks the $\root$ statistic.  This time using the
  $(\sigma,\nu_i)$ decomposition we have
  $$G(x,y) = 1 + xT + \frac{x}{y-1} T\bigl(G(x,y)-G(x,1)\bigr),
  $$
  where $T := tyG(1,y) / (1-tG(1,1))$. In this case the second and third summands correspond to \btrees\ with the $\rmod$ statistic equal to or greater than $1$, respectively.

  Since $F(x,y)$ and $G(x,y)$ satisfy the same equation with the same
  initial conditions $F(1,1)=G(1,1)$ being the generating function for
  \btrees, we must have $F(x,y)=G(x,y)$. On the other hand, by
  definition $F(x,y)=G(y,x)$. Thus, $F(x,y)=F(y,x)$ which proves
  Theorem \ref{thm} via the respective statistics on bicubic maps and
  \btrees.
\end{proof}

\section{Bicolored trees}

If we look at the parse tree of an expression of a \btree\ in terms of
$\sigma$ and $\nu_i$ (or $\rho$ and $\lambda_i$) we arrive at a new
tree. For instance, writing the tree from Figure~\ref{fig:tree} in
terms of $\sigma$ and $\nu_i$, as above, we arrive at
$$ %
\begin{tikzpicture}[yscale=0.5, xscale=0.3, semithick, baseline=20pt]
  \style;
  \node [wht] (r)      at ( 0,5) {};
  \node [sml] (r1)     at (-2,4) {};
  \node [sml] (r11)    at (-2,3) {};
  \node [wht] (r111)   at (-2,2) {};
  \node [sml] (r1111)  at (-3,1) {};
  \node [sml] (r11111) at (-4,0) {};
  \node [sml] (r11112) at (-3,0) {};
  \node [sml] (r11113) at (-2,0) {};
  \node [sml] (r1112)  at (-1,1) {};
  \node [sml] (r11121) at (-1,0) {};
  \node [sml] (r2)     at ( 2,4) {};
  \node [wht] (r21)    at ( 2,3) {};
  \node [sml] (r211)   at ( 1,2) {};
  \node [sml] (r2111)  at ( 1,1) {};
  \node [sml] (r21111) at ( 1,0) {};
  \node [sml] (r212)   at ( 3,2) {};
  \node [sml] (r2121)  at ( 3,1) {};
  \node [sml] (r21211) at ( 3,0) {};
  \draw    (r)     node[above left=0pt] {2}
        -- (r1)
        -- (r11)
        -- (r111)  node[left=1pt, yshift=2pt] {1}
        -- (r1111)
        -- (r11111)
(r1111) -- (r11112)
(r1111) -- (r11113)
 (r111) -- (r1112)
        -- (r11121)
    (r) -- (r2)
        -- (r21)   node[right=1pt, yshift=2pt] {2}
        -- (r211)
        -- (r2111)
        -- (r21111)
  (r21) -- (r212)
        -- (r2121)
        -- (r21211);
\end{tikzpicture}
$$
where an internal black node corresponds to $\sigma$ and a white node
labeled $i$ corresponds to $\nu_i$.
\newcommand{\T}{\mathcal{T}}

Let $\T$ denote the set of trees that can be obtained from \btrees\ in
this manner. Then it is not hard to see that $\T$ has the following
recursive characterization. A member of $\T$ is a rooted plane tree on
white and black nodes such that either the root is black and is
connected to a possibly empty list of trees in $\T$, or the root is
white, has a label $i$, is connected to exactly two trees $T_1,T_2\in
T$, and $1 \leq i\leq \kappa(T_2)$, where $\kappa$ is defined by
recursion: $\kappa$ of a leaf is $0$; $\kappa$ of a black node
connected to $T_1$, \dots ,$T_k$ is $1+\kappa(T_k)$; and $\kappa$ of a
white node labeled $i$, connected to $T_1$ and $T_2$, is
$i-1+\kappa(T_1)$. If, in addition, we define the \emph{weight}
of a tree in $\T$ to be the number of black nodes minus the number of
white nodes, then we have established that there is a one-to-one
correspondence between \btrees\ on $n$ nodes and trees in $\T$ of
weight $n$.

In the next section we shall define an endofunction on \btrees. One
way to understand this endofunction is that we map a \btree\ $T$ to a
\btree\ $T'$ if the $(\sigma, \nu_i)$ parse tree of $T$ is the same as
the $(\rho, \mu_i)$ parse tree of $T'$. We will prove that this
endofunction is an involution.

\section{An involution on \btrees}

The following three lemmas are immediate from the definitions of $\rho$,
$\mu_i$, $\sigma$ and $\nu_i$; they will be used in the proof of
Lemma~\ref{lemma:dual}.

\begin{lemma}\label{lemma:rho-nu}
  For all \btrees\ $T_1$, \dots, $T_k$ we have
  \begin{equation*}
    \rho(T_1,\dots,T_k) = \nu_1(\sigma(T_{k-1},\dots,T_1,\leaf), T_k).
  \end{equation*}
\end{lemma}
Note the similarity between Lemma~\ref{lemma:rho-nu} and
Definition~\ref{def:sigma}.

\begin{lemma}\label{lemma:nu-mu}
  Let $R$, $S$ and $T$ be \btrees. If $\open(R)=\root(S)=1$, and $T$
  is nontrivial, then, for integers  $i\geq 1$ and $j\geq 1$, we have
  $$\nu_{i+1}(R, \mu_{j}(S,T)) = \mu_{j+1}(S, \nu_{i}(R,T)).
  $$
\end{lemma}

\begin{lemma}\label{lemma:m1}
  Let $R$, $S$ and $T$ be \btrees. If $\root(R)=\open(R)=1$,
  then
  $$\mu_1(\nu_1(R,S),T) = \nu_1(\mu_1(R,T),S).
  $$
\end{lemma}

\begin{definition}\label{def:g}
  Let $T_1$, \dots, $T_k$, $S$ and $T$ be \btrees, and assume
  $\root(S)=1$. Define the map $g$ on \btrees\ of size $n$ by
  \begin{enumerate}
  \item $g(\leaf) = \leaf$;                                      \label{g:leaf}
  \item $g(\rho(T_1,\dots,T_k)) = \sigma(g(T_1),\dots,g(T_k))$;  \label{g:rho}
  \item $g(\mu_i(S,T)) = \nu_i(g(S), g(T))$.                     \label{g:mu}
  \end{enumerate}
\end{definition}
Note that there is a subtlety in this definition. In case
\eqref{g:mu}, we apply $\nu_i$ to $g(S)$, so we need to make sure that
$\open(g(S))=1$. But we are fine because, as $\root(S)=1$ then $S$ is
$\rho(T_1,\ldots,T_k)$, so to compute $g(S)$ we would use case
\eqref{g:rho} and the image under $\sigma$ of any sequence of trees has just one open node.

Here is an example of applying $g$:\vspace{-1ex}
$$\figtree\;\begin{array}{c}g\\\longmapsto\end{array}\;\quad
  \begin{tikzpicture}[xscale=0.37, yscale=0.5, semithick, baseline=(r)]
    \style;
    \node [sml] (r)       at ( 0,5) {};
    \node [sml] (r1)      at ( 0,4) {};
    \node [sml] (r11)     at (-2,3) {};
    \node [sml] (r111)    at (-2,2) {};
    \node [sml] (r1111)   at (-2,1) {};
    \node [sml] (r11111)  at (-2,0) {};
    \node [sml] (r12)     at ( 2,3) {};
    \node [sml] (r121)    at ( 2,2) {};
    \node [sml] (r1211)   at ( 2,1) {};
    \node [sml] (r12111)  at (0.5,0) {};
    \node [sml] (r12112)  at ( 2,0) {};
    \node [sml] (r12113)  at (3.5,0) {};
    \draw (r)      node[above=2pt] {2} --
          (r1)     node[above right=-1pt] {1} --
          (r11)    node[left=2pt] {2} --
          (r111)   node[left=2pt] {1} --
          (r1111)  node[left=2pt] {0} --
          (r11111) node[left=2pt] {0}
          (r1) --
          (r12)    node[right=2pt] {0} --
          (r121)   node[right=2pt] {0} --
          (r1211)  node[right=2pt, yshift=1pt] {1} --
          (r12111) node[below=2pt] {0}
          (r1211) --
          (r12112) node[below=2pt] {0}
          (r1211) --
          (r12113) node[below=2pt] {0};
  \end{tikzpicture}
$$
For a larger example see the appendix, where two \btrees\ (and
associated bicubic maps) corresponding to each other under $g$ are
given.

\begin{lemma}\label{lemma:dual}
  If $T_1$, \dots, $T_k$, $S$ and $T$ are \btrees, and $\open(S)=1$,
  then
  \begin{enumerate}
  \item\label{dual-1} $g(\sigma(T_1,\dots,T_k)) = \rho(g(T_1),\dots,g(T_k))$;
  \item\label{dual-2} $g(\nu_i(S,T)) = \mu_i(g(S), g(T))$.
  \end{enumerate}
\end{lemma}

\begin{proof}
  We have
  \begin{align*}
    g(\sigma(T_1,\dots,T_k))
    &= g(\mu_1(\rho(T_{k-1},\dots,T_1,\leaf), T_k))
    &&\text{by~Definition~\ref{def:sigma}}\\
    &= \nu_1(g(\rho(T_{k-1},\dots,T_1,\leaf)), g(T_k))
    &&\text{by Definition~\ref{def:g}}\\
    &= \nu_1(\sigma(g(T_{k-1}),\dots,g(T_1),\leaf)), g(T_k))
    &&\text{by Definition~\ref{def:g}}\\
    &= \rho(g(T_1),\dots,g(T_k))
    &&\text{by Lemma~\ref{lemma:rho-nu}}
  \end{align*}
  which proves \eqref{dual-1}. To prove \eqref{dual-2} we first note
  that $\root(\nu_i(S,T)) = 1$ if, and only if, $\root(S)=1$ and
  $i=1$.  Accordingly, the proof of \eqref{dual-2} will be split into
  three cases:
  \begin{itemize}
  \item[(a)] $i=1$ and $\root(S)=1$;
  \item[(b)] $i=1$ and $\root(S)>1$;
  \item[(c)] $i>1$.
  \end{itemize}

  Case (a): By assumption, $\open(S)=1$; if also $\root(S)=1$, then
  $S$ must be of the form $S=\sigma(S_1,\dots,S_{\ell-1},\leaf)$ for
  some \btrees\ $S_1$, \dots, $S_{\ell-1}$, and thus
  \begin{align*}
    \nu_1(S,T)
    &= \nu_1(\sigma(S_1,\dots,S_{\ell-1},\leaf), T)\\
    &= \rho(S_{\ell-1},\dots,S_{1},T).
    &&\text{by Lemma~\ref{lemma:rho-nu}}
  \end{align*}
  Therefore,
  \begin{align*}
    g(\nu_1(S,T))
    &= \sigma(g(S_{\ell-1}),\dots,g(S_1),g(T))
    &&\text{by Definition~\ref{def:g}}\\
    &= \mu_1(\rho(g(S_1),\dots,g(S_{\ell-1}), \leaf),g(T))
    &&\text{by Definition~\ref{def:sigma}}\\
    &= \mu_1(g(\sigma(S_1,\dots,S_{\ell-1},\leaf)), g(T))
    &&\text{by~\eqref{dual-1}}\\
    &= \mu_1(g(S),g(T)).
  \end{align*}

  Case (b): Since $\root(S)>1$ there
  are \btrees\ $U$ and $V$, and an integer $j$, such that $\root(U)=1$,
  $V$ is nontrivial, and $S=\mu_j(U,V)$. By assumption $\open(S)=1$. Moreover, 
 item~\eqref{2d} from Lemma~\ref{lem:behaviour} implies that 
 $\open(U)=1$ and $j=1$; thus we can use
  Lemma~\ref{lemma:m1}. The proof now proceeds by structural induction
  (the base case is trivial):
  \begin{align*}
    g(\nu_1(S,T))
    &= g(\nu_1(\mu_1(U,V),T)) \\
    &= g(\mu_1(\nu_1(U,T),V))
    &&\text{by Lemma~\ref{lemma:m1}}\\
    &= \nu_1(g(\nu_1(U,T)),g(V))
    &&\text{by Definition~\ref{def:g}}\\
    &= \nu_1(\mu_1(g(U),g(T)),g(V))
    &&\text{by induction}\\
    &= \mu_1(\nu_1(g(U),g(V)),g(T))
    &&\text{by Lemma~\ref{lemma:m1}}\\
    &= \mu_1(g(\mu_1(U,V)),g(T))
    &&\text{by Definition~\ref{def:g}}\\
    &= \mu_1(g(S),g(T)).
  \end{align*}

  Case (c): If $i>1$, then $\root(T)>1$ and we can write
  $T=\mu_j(U,V)$ for some \btrees\ $U$ and $V$ with $\root(U)=1$ and
  $V$ nontrivial. We can now proceed by either using structural
  induction or induction on $i$, the base case $i=1$ being provided by
  cases (a) and (b) above:
  \begin{align*}
    g(\nu_i(S,T))
    &= g(\nu_i(S,\mu_j(U,V)))
    && \\
    &= g(\mu_{j+1}(U,\nu_{i-1}(S,V)))
    &&\text{by Lemma~\ref{lemma:nu-mu}}\\
    &= \nu_{j+1}(g(U),g(\nu_{i-1}(S,V)))
    &&\text{by Definition~\ref{def:g}}\\
    &= \nu_{j+1}(g(U),\mu_{i-1}(g(S),g(V)))
    &&\text{by induction}\\
    &= \mu_{i}(g(S),\nu_{j}(g(U),g(V)))
    &&\text{by Lemma~\ref{lemma:nu-mu}}\\
    &= \mu_{i}(g(S), g(\mu_{j}(U,V)))
    &&\text{by Definition~\ref{def:g}}\\
    &= \mu_{i}(g(S), g(T))
  \end{align*}
  which concludes the proof.
\end{proof}

\begin{theorem}\label{thm:g-is-an-involution}
  The map $g$ is an involution.
\end{theorem}

\begin{proof}
  We use induction on size. The base case $g^2(\leaf)=\leaf$ is
  trivial. For the induction step we have
  \begin{align*}
    g^2(\rho(T_1,\dots,T_k))
    &= g(\sigma(g(T_1),\dots,g(T_k)))
    &&\text{by Definition~\ref{def:g}}\\
    &= \rho(g^2(T_1),\dots,g^2(T_k))
    &&\text{by Lemma~\ref{lemma:dual}}\\
    &= \rho(T_1,\dots,T_k)
    &&\text{by induction}\\
    \intertext{and}
    g^2(\mu_i(S,T))
    &= g(\nu_i(g(S), g(T)))
    &&\text{by Definition~\ref{def:g}}\\
    &= \mu_i(g^2(S), g^2(T))
    &&\text{by Lemma~\ref{lemma:dual}}\\
    &= \mu_i(S, T)
    &&\text{by induction}
  \end{align*}
  which concludes the proof.
\end{proof}

\begin{theorem}\label{thm:eqd}
  On $\beta(0,1)$-trees with $n$ nodes, the pair of statistics $(\root,\open)$ has
  the same joint distribution as the pair
  $(\open,\root)$. Equivalently,
  $$\sum_{T}x^{\root(T)}y^{\open(T)} = \sum_{T}x^{\open(T)}y^{\root(T)},
  $$
  where the sum is over all \btrees\ with $n$ nodes.
\end{theorem}

\begin{proof}
  Using induction we shall now prove that $\root(g(U)) = \open(U)$ for
  each \btree\ $U$.  The base case is plain. For the induction step,
  assume that $T_1$, \dots, $T_k$, $S$ and $T$ are \btrees,
  $\root(S)=1$, and that $T$ is nontrivial. We have
  \begin{align*}
    \root(g(\rho(T_1,\dots,T_k)))
    &= \root(\sigma(g(T_1),\dots,g(T_k)))
    && \text{by Definition~\ref{def:g}}\\
    &= 1 + \root(g(T_k))
    && \text{by \eqref{2a} from Lemma~\ref{lem:behaviour}}\\
    &= 1 + \open(T_k)
    && \text{by induction}\\
    &= \open(\rho(T_1,\dots,T_k)).
    && \text{by \eqref{1b} from Lemma~\ref{lem:behaviour}}
    \intertext{Also,}
    \root(g(\mu_i(S,T)))
    &= \root(\nu_i(g(S),g(T)))
    && \text{by Definition~\ref{def:g}}\\
    &= i - 1 + \root(g(S))
    && \text{by \eqref{1c} from Lemma~\ref{lem:behaviour}}\\
    &= i - 1 + \open(S)
    && \text{by induction}\\
    &= \open(\mu_i(S,T)).
    && \text{by \eqref{2d} from Lemma~\ref{lem:behaviour}}
  \end{align*}
  Since $g$ is an involution it follows that $\open(g(T)) =
  \root(T)$ as well, which concludes the proof.
\end{proof}

\begin{corollary}\label{corollary11}
  On $\beta(0,1)$-trees with $n$ nodes, the pair of statistics $(\root,\rmod)$ has
  the same joint distribution as the pair
  $(\rmod,\root)$. Equivalently,
  $$\sum_{T}x^{\root(T)}y^{\rmod(T)} = \sum_{T}x^{\rmod(T)}y^{\root(T)},
  $$
  where both sums are over all \btrees\ with $n$ nodes.
\end{corollary}

\begin{proof}
  This is a direct consequence of Lemma~\ref{lemma:rmod-open} and
  Theorem~\ref{thm:eqd}.
\end{proof}

Our second proof of Theorem \ref{thm} now follows from Corollary
\ref{corollary11} through the correspondence between bicubic maps and
\btrees.

\begin{definition}
  Let $C_n=\binom{2n}{n}/(n+1)$ denote the $n$th Catalan
  number. Define
  $$a(n) = 2^{n-1}C_n.
  $$
  This is sequence A003645 in OEIS~\cite{oeis}.
\end{definition}

By computing the number of trees fixed by $g$, for $n\leq 12$, we
arrive at the following conjecture.
 
\begin{conjecture}\label{conj:fp}
  For $n>1$, the number of \btrees\ on $n$ nodes fixed under $g$ is
  $a(\floor{n/2})$. This sequence starts
  $1,1,4,4,20,20,112,112,672,672,4224,4224,\dots$
\end{conjecture}
% [ length $ filter (\t -> asN t == asR t) $ treesT n | n <- [1..12] ]
% [1,1,1,4,4,20,20,112,112,672,672,4224]

The number of fixed points under the involution $h$ on
$\beta(1,0)$-trees (introduced in \cite{CKS,CKS0}) was found in
\cite{KM}. These numbers also count {\em self-dual rooted
  non-separable planar maps} \cite{KMN}. However, we were not able to
exploit the ideas to count fixed points under $h$ in order to prove
Conjecture \ref{conj:fp}, because the involution $g$ is more complex,
and in general, $\beta(0,1)$-trees are more complex than
$\beta(1,0)$-trees.

\begin{proposition}[Tutte, Koganov, Liskovets and Walsh]
\label{prop:TKLW}
  The number of bicubic maps on $2n$ vertices with one distinguished
  $1$-colored face is $a(n)$.
\end{proposition}

\begin{proof}
  Koganov, Liskovets and Walsh \cite[Proposition 3.1]{KLW} showed that
  the number of rooted eulerian planar maps with $n$ edges and a
  distinguished vertex is given by the formula $a(n)$. Tutte's well-known
``trinity mapping'' sends eulerian planar maps with $n$ edges
  to bicubic maps with $2n$ vertices. It is easy to see that under the
  same mapping vertices are sent to $1$-colored faces.
\end{proof}

\begin{proposition}
  The number of \btrees\ on $n+1$ nodes with one distinguished
  excessive node is $a(n)$.
\end{proposition}

\begin{proof}
  This is a direct consequence of Propositions~\ref{prop:exc} and \ref{prop:TKLW}.
\end{proof}

In light of this last proposition we can reformulate
Conjecture~\ref{conj:fp} as follows.

\begin{conjecture}
  There is a bijection between \btrees\ on $n$ nodes fixed under $g$
  and \btrees\ on $\floor{n/2}+1$ nodes with one distinguished
  excessive node.
\end{conjecture}

We close this paper by making an additional conjecture.

\begin{conjecture}\label{conj2}
  The two pairs of statistics $(\root,\rzero)$ and $(\rmod,\sub)$
  are jointly equidistributed on $\beta(0,1)$-trees.
\end{conjecture}

We have verified Conjecture~\ref{conj2} for {\btrees} on at most 11
nodes. This conjecture will imply, via the bijection described in
Section~\ref{sec:maps-trees}, that the two pairs of statistics
$(\rootM, \rzeroM)$ and $(\rmodM, \subM)$ are jointly equidistributed
on bicubic maps.

\section{Acknowledgments}
We wish to thank the anonymous referee (of an earlier version of this
paper) who came up with the generating function based proof of Theorem
\ref{thm}.

The third author was supported by the Spanish and Catalan
  governments under projects MTM2011--24097 and DGR2009--SGR1040.

\section{Appendix}

Below are some examples of the map $\psi$ from bicubic maps to
\btrees. The image of each large map at the top is the tree below it, and
for each smaller map, its image is the subtree consisting in the edge
next to it and all the edges below, with the root label adjusted if
necessary.

Also, the two trees are the image of each other under the involution
$g$. For the first tree and map, $\exc(T) = \one(M) = 6$, $\root(T) =
\rootM(M)=4$, $\rmod(T) = \rmodM(M) =2$, $\rzero(T) = \rzeroM(M) =1$,
and $\sub(T) = \subM(M)=4$. For the second tree and map,
$\exc(T) = \one(M) = 6$, $\root(T) = \rootM(M)=4$, $\rmod(T) =
\rmodM(M) =2$, $\rzero(T) = \rzeroM(M) = 3$, and $\sub(T) =
\subM(M)=4$.
\begin{figure}
  $$\includegraphics[width=12cm]{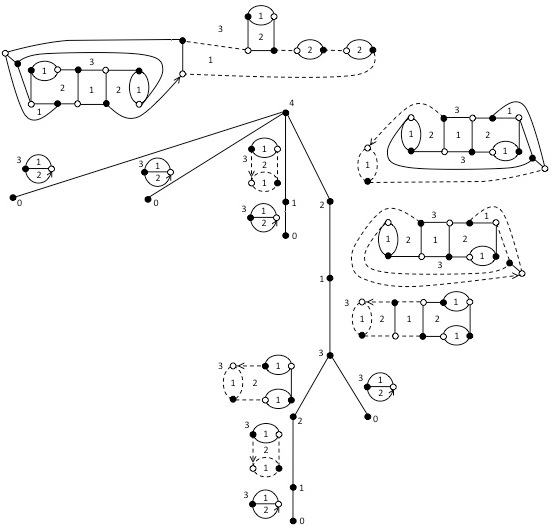}
  $$
\end{figure}

\begin{figure}
  $$\includegraphics[width=12cm]{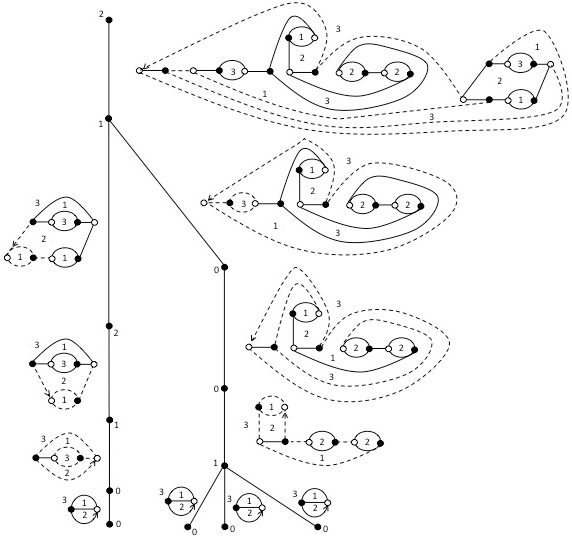}
  $$
\end{figure}

\end{document}